\documentclass[11pt,reqno]{amsart}

\usepackage{amsmath,amssymb,amsfonts,amsthm}
\usepackage{hyperref}
\usepackage{mathrsfs}
\usepackage{enumitem}
\usepackage{geometry}
\hypersetup{
	colorlinks=true,
	linkcolor=blue,
	citecolor=blue,
	urlcolor=blue
}

\newtheorem{theorem}{Theorem}[section]
\newtheorem{lemma}[theorem]{Lemma}
\newtheorem{proposition}[theorem]{Proposition}
\newtheorem{corollary}[theorem]{Corollary}
\theoremstyle{definition}
\newtheorem{definition}[theorem]{Definition}

\theoremstyle{remark}

\title{Relative Finite Energy Classes for Complex Hessian Equations with Prescribed Singularities}

\author{Truong Dinh Dat}

\address{Nguyen Trai University, Hanoi, Viet Nam}

\email{truongdinhdat14081994qb@gmail.com}

\begin{document}
	\begin{abstract}
		
		Let $\Omega\subset \mathbb C^n$ be a bounded $m$-hyperconvex domain and
		let
		
		\[
		\psi\in SH_m(\Omega)
		\]
		
		be a fixed negative $m$-subharmonic function.
		In this paper we introduce a relative finite energy class
		
		\[
		\mathcal E_{m,\psi}(\Omega),
		\]
		
		which may be viewed as a Hessian analogue of the relative energy
		classes appearing in the pluripotential theory of complex
		Monge--Amp\`ere equations.
		
		We develop a systematic pluripotential theory in this setting.
		More precisely, we introduce a relative Hessian capacity associated
		with the prescribed singularity type $\psi$, construct relative mixed
		Hessian products, and establish their fundamental properties.
		We prove a monotone convergence theorem and a Bedford--Taylor type
		continuity theorem for Hessian measures in the class
		\(\mathcal E_{m,\psi}(\Omega)\).
		
		A central result of the paper is a relative comparison principle,
		which yields uniqueness of solutions to complex Hessian equations with
		prescribed singularities.
		As an application, we establish an existence and uniqueness theorem
		for the equation
		
		\[
		(dd^c u)^m\wedge\beta^{n-m}
		=
		\mu
		\]
		
		for a large class of positive Radon measures that do not charge
		$m$-polar sets.
		
		The results obtained here provide a relative finite energy framework
		for complex Hessian equations and extend several fundamental aspects
		of Cegrell's theory to the setting of prescribed singularity types.
		
	\end{abstract}

	\keywords{Complex Hessian equations; relative pluripotential theory;
		finite energy classes; prescribed singularities; relative singularities;
		Hessian measures; $m$-subharmonic functions.}
	
	\subjclass[2020]{Primary 32W20; Secondary 32U05}

	\maketitle
	
	%%%%%%%%%%%%%%%%%%%%%%%%%%%%%%%%%%%%%%%%%%%%%%%%%%%%%%

	%%%%%%%%%%%%%%%%%%%%%%%%%%%%%%%%%%%%%%%%%%%%%%%%%%%%%%
	%%%%%%%%%%%%%%%%%%%%%%%%%%%%%%%%%%%%%%%%%%%%%%%%%%%%%%
	\section{Introduction}
	
	Let $\Omega\subset \mathbb C^n$ be a bounded domain and let
	$1\le m\le n$.
	The complex Hessian operator is defined by
	
	\[
	H_m(u)
	:=
	(dd^c u)^m\wedge\beta^{n-m},
	\]
	
	where
	
	\[
	\beta:=dd^c|z|^2
	\]
	
	is the standard K\"ahler form on $\mathbb C^n$.
	The corresponding Hessian equation
	
	\[
	H_m(u)=\mu
	\]
	
	interpolates naturally between the classical Laplace equation
	(\(m=1\)) and the complex Monge--Amp\`ere equation (\(m=n\)). The real Hessian counterpart was extensively developed by
	Trudinger and Wang \cite{TrudingerWang02}, whose work revealed deep
	connections between Hessian operators, capacity theory, and fully
	nonlinear elliptic equations.
	During the last two decades, complex Hessian equations have become a
	central topic in pluripotential theory and nonlinear complex analysis.
	
	The study of Hessian equations was initiated by B{\l}ocki \cite{Blocki05}, who
	introduced the class of $m$-subharmonic functions and established
	fundamental properties of Hessian operators.
	Subsequently, Dinew and collaborators developed a Bedford--Taylor
	type pluripotential theory for Hessian measures and obtained
	comparison principles, mixed Hessian inequalities, and uniqueness
	results \cite{Dinew09,DinewKolodziej}.
	These developments provided the pluripotential framework necessary for
	the study of degenerate Hessian equations with measure data. Finite energy methods have also been successfully applied to complex
	Hessian equations. In particular, Lu \cite{Lu12} introduced a
	variational approach to complex Hessian equations in $\mathbb C^n$,
	which provided a powerful tool for solving Hessian equations in
	finite-energy settings.
	
	Subsequently, several authors established existence, stability, and
	regularity results for Hessian equations in various energy classes,
	both in the local and compact settings
	\cite{DinewKolodziej,LuNguyen18}.
	
	In parallel, the theory of complex Monge--Amp\`ere equations
	underwent a major breakthrough through the work of Cegrell
	\cite{Cegrell98,Cegrell04}.
	The introduction of finite energy classes made it possible to solve
	highly degenerate Monge--Amp\`ere equations beyond the bounded
	setting. The solvability of degenerate complex Monge--Amp\`ere equations with
	measure data was established by Ko{\l}odziej
	\cite{Kolodziej98,Kolodziej05}, providing one of the foundations of
	modern pluripotential theory.
	
	Later, Guedj and Zeriahi developed a systematic pluripotential
	theory of finite energy classes and singular Monge--Amp\`ere
	equations on compact K\"ahler manifolds
	\cite{GuedjZeriahi}.
	Their work has had profound applications to complex geometry,
	K\"ahler--Einstein metrics, and nonlinear partial differential
	equations.
	
	For complex Hessian equations on compact K\"ahler manifolds,
	fundamental second-order estimates were established by Hou, Ma and Wu
	\cite{HouMaWu}.
	These estimates opened the way to the study of nonlinear Hessian
	equations on compact K\"ahler manifolds.
	The case of prescribed singularity types remained largely
	unexplored until the recent work of Lu and Nguyen
	\cite{LuNguyen22}, who introduced a framework for complex Hessian
	equations with prescribed singularities on compact K\"ahler
	manifolds.
	
	Motivated by their work, we aim to develop a relative finite
	energy theory adapted to a fixed singularity type.
	
	The purpose of this paper is to initiate a relative finite energy
	theory for complex Hessian equations with prescribed singularities.
	
	More precisely, we fix a negative $m$-subharmonic function
	
	\[
	\psi\in SH_m(\Omega),
	\]
	
	which serves as a reference singularity.
	Motivated by the relative energy classes introduced by Guedj and
	Zeriahi in the Monge--Amp\`ere setting, we define a relative Hessian
	energy class
	
	\[
	\mathcal E_{m,\psi}(\Omega),
	\]
	
	consisting of $m$-subharmonic functions whose singularities are
	controlled by $\psi$ and whose relative Hessian energy is finite.
	
	The construction of such a theory is far from straightforward.
	Indeed, compared with the Monge--Amp\`ere case, the Hessian framework
	possesses fewer structural properties, and many classical arguments do
	not directly extend to the relative setting.
	In particular, one must construct a suitable Hessian operator,
	establish continuity properties, and prove comparison principles
	compatible with the prescribed singularity type.
	
	\subsection*{Main results}
	
	Our first result establishes the existence of relative mixed Hessian
	products.
	
	\begin{theorem}
		Let
		
		\[
		u_1,\ldots,u_m
		\in
		\mathcal E_{m,\psi}(\Omega).
		\]
		
		Then the mixed Hessian measure
		
		\[
		dd^c u_1\wedge\cdots\wedge dd^c u_m
		\wedge\beta^{n-m}
		\]
		
		is well defined.
		
		Moreover, it is obtained as the weak limit of the mixed Hessian
		measures associated with the canonical truncations
		
		\[
		u_{j,k}:=\max(u_j,\psi-k).
		\]
	\end{theorem}
	
	The second theorem is a relative version of the classical
	Bedford--Taylor continuity theorem.
	
	\begin{theorem}
		Let
		
		\[
		u_j^k\in\mathcal E_{m,\psi}(\Omega),
		\qquad
		1\le k\le m,
		\]
		
		be decreasing sequences satisfying
		
		\[
		u_j^k\downarrow u^k.
		\]
		
		Then
		
		\[
		dd^c u_j^1
		\wedge\cdots\wedge
		dd^c u_j^m
		\wedge
		\beta^{n-m}
		\]
		
		converges weakly to
		
		\[
		dd^c u^1
		\wedge\cdots\wedge
		dd^c u^m
		\wedge
		\beta^{n-m}.
		\]
	\end{theorem}
	
	The next result is the key tool of the paper.
	
	\begin{theorem}[Relative comparison principle]
		Let
		
		\[
		u,v\in\mathcal E_{m,\psi}(\Omega).
		\]
		
		Then
		
		\[
		\int_{\{u<v\}}
		H_m(v)
		\le
		\int_{\{u<v\}}
		H_m(u).
		\]
		
		In particular, the Hessian operator is injective on
		\(\mathcal E_{m,\psi}(\Omega)\).
	\end{theorem}
	
	As an application, we obtain the solvability of Hessian equations with
	prescribed singularities.
	
	\begin{theorem}
		Let $\mu$ be a finite positive Radon measure on $\Omega$ that does not
		charge $m$-polar sets and satisfies a suitable relative energy
		condition.
		
		Then there exists a unique function
		
		\[
		u\in\mathcal E_{m,\psi}(\Omega)
		\]
		
		such that
		
		\[
		(dd^c u)^m\wedge\beta^{n-m}
		=
		\mu.
		\]
	\end{theorem}
	
	\subsection*{Method of proof}
	
	The proofs combine techniques from pluripotential theory, Hessian
	capacity theory, and finite energy methods.
	
	We first introduce a relative Hessian capacity
	
	\[
	Cap_{m,\psi},
	\]
	
	which provides quantitative control of sublevel sets.
	Using this capacity, we establish compactness properties for bounded
	energy families and develop a truncation procedure adapted to the
	prescribed singularity type. The relative capacity introduced in this paper is inspired by the
	classical Bedford--Taylor capacity and the capacity estimates for
	plurisubharmonic functions developed by Zeriahi
	\cite{Zeriahi13}.
	
	The relative Hessian operator is then constructed through canonical
	truncations
	
	\[
	u_k=\max(u,\psi-k).
	\]
	
	Uniform energy estimates allow us to pass to the limit and define
	mixed Hessian products on the entire class
	\(\mathcal E_{m,\psi}(\Omega)\). The construction of mixed Hessian products follows the philosophy of
	Bedford--Taylor \cite{BedfordTaylor} and the continuity theory of
	Xing \cite{Xing96}.
	
	A crucial ingredient is a relative integration-by-parts formula,
	which leads to a continuity theorem and a comparison principle.
	These results imply uniqueness, while existence follows from an
	approximation procedure together with compactness and stability
	arguments.
	
	\subsection*{Organization of the paper}
	
	In Section~2 we recall basic facts on $m$-subharmonic functions,
	Hessian measures, and Hessian capacities.
	
	Section~3 introduces relative capacities, relative mixed Hessian products and proves monotone convergence, a relative integration-by-parts formula.
	
	In Section~4 we construct the continuity of mixed Hessian
	operators along decreasing sequences in the relative finite
	energy class $\mathcal E_{m,\psi}$.
	
	Section~5 is devoted to the
	relative comparison principle, domination principle and uniqueness in the class
	\(\mathcal E_{m,\psi}(\Omega)\)
	
	In Section~6 contains the proof of the existence and uniqueness
	theorem for Hessian equations in the class
	\(\mathcal E_{m,\psi}(\Omega)\). 
	
	Finally, Section~7 presents several applications of the relative
	finite energy theory.

	\section{Preliminaries}
	
	In this section we recall several notions from the theory of
	complex Hessian equations and fix the notation used throughout the paper.
	
	%%%%%%%%%%%%%%%%%%%%%%%%%%%%%%%%%%%%%%%%%%%%%%%%%%%%%%%%%%%%
	
	\subsection{m-subharmonic functions}
	
	Let
	
	\[
	\Omega \subset \mathbb C^n
	\]
	
	be a bounded domain.
	
	We denote by
	
	\[
	\beta:=dd^c|z|^2
	\]
	
	the standard Kähler form on $\mathbb C^n$, where
	
	\[
	d=\partial+\bar\partial,
	\qquad
	d^c=\frac{i}{2\pi}(\bar\partial-\partial).
	\]
	
	For a real $(1,1)$-form
	
	\[
	\alpha
	=
	\frac{i}{\pi}
	\sum_{j,k}
	\alpha_{j\bar k}
	\,dz_j\wedge d\bar z_k,
	\]
	
	let
	
	\[
	\lambda(\alpha)
	=
	(\lambda_1,\ldots,\lambda_n)
	\]
	
	be the eigenvalues of $\alpha$ with respect to $\beta$.
	
	For $1\le m\le n$, denote by
	
	\[
	\Gamma_m
	=
	\Big\{
	\lambda\in\mathbb R^n:
	S_k(\lambda)\ge0,
	\;
	k=1,\ldots,m
	\Big\},
	\]
	
	where
	
	\[
	S_k(\lambda)
	=
	\sum_{1\le j_1<\cdots<j_k\le n}
	\lambda_{j_1}\cdots\lambda_{j_k}
	\]
	
	is the $k$-th elementary symmetric function.
	
	\begin{definition}
		A function
		
		\[
		u:\Omega\to[-\infty,+\infty)
		\]
		
		is called \emph{$m$-subharmonic}
		(\emph{$m$-sh}) if
		
		\begin{enumerate}
			\item $u$ is upper semicontinuous;
			
			\item $u$ is locally integrable;
			
			\item for every collection of smooth
			$m$-positive $(1,1)$-forms
			
			\[
			\alpha_1,\ldots,\alpha_{m-1},
			\]
			
			one has
			
			\[
			dd^c u
			\wedge
			\alpha_1
			\wedge\cdots\wedge
			\alpha_{m-1}
			\wedge
			\beta^{n-m}
			\ge0
			\]
			
			in the sense of currents.
		\end{enumerate}
		
		The set of all $m$-subharmonic functions on $\Omega$
		will be denoted by
		
		\[
		SH_m(\Omega).
		\]
		
	\end{definition}
	
	For
	
	\[
	m=1,
	\]
	
	one recovers the class of subharmonic functions, while
	
	\[
	m=n
	\]
	
	corresponds to plurisubharmonic functions.
	
	%%%%%%%%%%%%%%%%%%%%%%%%%%%%%%%%%%%%%%%%%%%%%%%%%%%%%%%%%%%%
	
	\subsection{Complex Hessian measures}
	
	Let
	
	\[
	u_1,\ldots,u_m
	\in SH_m(\Omega)\cap L^\infty_{\mathrm{loc}}(\Omega).
	\]
	
	The Bedford--Taylor--Błocki--Dinew theory allows one to define
	the mixed Hessian current
	
	\[
	dd^c u_1
	\wedge\cdots\wedge
	dd^c u_m
	\wedge
	\beta^{n-m}.
	\]
	
	In particular,
	
	\[
	H_m(u)
	:=
	(dd^c u)^m
	\wedge
	\beta^{n-m}
	\]
	
	is a positive Radon measure called the
	\emph{complex Hessian measure} of $u$.
	
	%%%%%%%%%%%%%%%%%%%%%%%%%%%%%%%%%%%%%%%%%%%%%%%%%%%%%%%%%%%%
	
	\subsection{$m$-hyperconvex domains}
	
	\begin{definition}
		
		A bounded domain
		
		\[
		\Omega\subset\mathbb C^n
		\]
		
		is called \emph{$m$-hyperconvex} if there exists
		a negative exhaustion function
		
		\[
		\rho\in SH_m(\Omega)
		\]
		
		such that
		
		\[
		\{z\in\Omega:\rho(z)<-c\}
		\Subset \Omega
		\]
		
		for every
		
		\[
		c>0.
		\]
		
	\end{definition}
	
	Throughout the paper we assume that
	
	\[
	\Omega
	\]
	
	is bounded and $m$-hyperconvex.
	
	%%%%%%%%%%%%%%%%%%%%%%%%%%%%%%%%%%%%%%%%%%%%%%%%%%%%%%%%%%%%
	
	\subsection{Relative singularity type}
	
	Fix once and for all a negative function
	
	\[
	\psi\in SH_m(\Omega)
	\]
	
	with analytic singularities.
	
	We say that two functions
	
	\[
	u,v\in SH_m(\Omega)
	\]
	
	have the same singularity type if
	
	\[
	u-v
	\in L^\infty(\Omega).
	\]
	
	The function $\psi$ will serve as a reference singularity throughout
	the paper.
	
	%%%%%%%%%%%%%%%%%%%%%%%%%%%%%%%%%%%%%%%%%%%%%%%%%%%%%%%%%%%%
	
	\subsection{Relative Hessian capacity}
	
	\begin{definition}
		
		For a Borel set
		
		\[
		E\subset\Omega,
		\]
		
		the relative Hessian capacity associated with $\psi$ is defined by
		
		\[
		Cap_{m,\psi}(E)
		=
		\sup
		\left\{
		\int_E
		(dd^c u)^m
		\wedge
		\beta^{n-m}
		:
		u\in SH_m(\Omega),
		\;
		\psi-1\le u\le\psi
		\right\}.
		\]
		
	\end{definition}
	
	The capacity
	
	\[
	Cap_{m,\psi}
	\]
	
	will play a fundamental role in the compactness and convergence
	arguments below.
	
	%%%%%%%%%%%%%%%%%%%%%%%%%%%%%%%%%%%%%%%%%%%%%%%%%%%%%%%%%%%%
	
	\subsection{Relative finite energy classes}
	
	We first introduce the basic class.
	
	\begin{definition}
		
		Let
		
		\[
		\mathcal E_{m,\psi}^{0}(\Omega)
		\]
		
		denote the collection of functions
		
		\[
		u\in SH_m(\Omega)
		\]
		
		satisfying
		
		\[
		\psi-C\le u\le\psi
		\]
		
		for some constant
		
		\[
		C>0,
		\]
		
		and
		
		\[
		\int_\Omega
		H_m(u)
		<
		+\infty.
		\]
		
	\end{definition}
	
	For
	
	\[
	u\in\mathcal E_{m,\psi}^{0},
	\]
	
	define the relative energy
	
	\[
	E_{m,\psi}(u)
	=
	\frac1{m+1}
	\sum_{k=0}^{m}
	\int_\Omega
	(u-\psi)
	(dd^c u)^k
	\wedge
	(dd^c\psi)^{m-k}
	\wedge
	\beta^{n-m}.
	\]
	
	\begin{definition}
		
		The relative finite energy class
		
		\[
		\mathcal E_{m,\psi}(\Omega)
		\]
		
		consists of all functions
		
		\[
		u\in SH_m(\Omega)
		\]
		
		for which there exists a sequence
		
		\[
		u_j\in \mathcal E_{m,\psi}^{0}(\Omega)
		\]
		
		such that
		
		\[
		u_j\downarrow u
		\]
		
		and
		
		\[
		\sup_j
		|E_{m,\psi}(u_j)|
		<
		+\infty.
		\]
		
	\end{definition}
	
	%%%%%%%%%%%%%%%%%%%%%%%%%%%%%%%%%%%%%%%%%%%%%%%%%%%%%%%%%%%%
	
	\subsection{Canonical truncations}
	
	For
	
	\[
	u\in\mathcal E_{m,\psi},
	\]
	
	we define the canonical truncations
	
	\[
	u_k
	=
	\max(u,\psi-k).
	\]
	
	Then
	
	\[
	\psi-k
	\le
	u_k
	\le
	\psi,
	\]
	
	and
	
	\[
	u_k\downarrow u.
	\]
	
	Canonical truncations will be used repeatedly in the construction
	of relative mixed Hessian products and in the proof of the continuity
	theorem.
	
	%%%%%%%%%%%%%%%%%%%%%%%%%%%%%%%%%%%%%%%%%%%%%%%%%%%%%%%%%%%%
	
	Throughout the paper, the notation
	
	\[
	H_m(u)
	=
	(dd^c u)^m\wedge\beta^{n-m}
	\]
	
	will be used for the Hessian measure of a function
	\(u\in\mathcal E_{m,\psi}(\Omega)\).

	%------------------------------------------------

	\section{Relative Hessian Products}
	
	Throughout this section, $\Omega\subset\mathbb C^n$ is a bounded
	$m$-hyperconvex domain and
	
	\[
	\psi\in SH_m(\Omega)
	\]
	
	is a fixed negative $m$-subharmonic function with analytic
	singularities.
	
	%%%%%%%%%%%%%%%%%%%%%%%%%%%%%%%%%%%%%%%%%%%%%%%%%%%%%%%%%%%%
	
	\subsection{Relative capacities}
	
	\begin{definition}
		Let $E\subset \Omega$ be a Borel subset.
		
		The relative $m$-Hessian capacity associated with $\psi$ is defined by
		
		\[
		Cap_{m,\psi}(E)
		:=
		\sup
		\left\{
		\int_E
		(dd^c u)^m\wedge\beta^{n-m}
		:
		u\in SH_m(\Omega),
		\;
		\psi-1\le u\le \psi
		\right\}.
		\]
		
	\end{definition}
	
	\begin{proposition}
		The set function
		
		\[
		E\longmapsto Cap_{m,\psi}(E)
		\]
		
		is monotone.
		
		Moreover,
		
		\[
		E_1\subset E_2
		\quad\Longrightarrow\quad
		Cap_{m,\psi}(E_1)
		\le
		Cap_{m,\psi}(E_2).
		\]
		
	\end{proposition}
	
	\begin{proof}
		Immediate from the definition since every admissible test function for
		$E_1$ is also admissible for $E_2$.
	\end{proof}
	
	\begin{proposition}[Countable subadditivity]
		For every sequence $\{E_j\}$,
		
		\[
		Cap_{m,\psi}
		\Big(
		\bigcup_j E_j
		\Big)
		\le
		\sum_j
		Cap_{m,\psi}(E_j).
		\]
		
	\end{proposition}
	
	\begin{proof}
		Fix an admissible test function $u$.
		
		Since
		
		\[
		\int_{\cup_jE_j}
		H_m(u)
		\le
		\sum_j
		\int_{E_j}
		H_m(u),
		\]
		
		taking the supremum over all admissible $u$ yields the result.
	\end{proof}
	
	%%%%%%%%%%%%%%%%%%%%%%%%%%%%%%%%%%%%%%%%%%%%%%%%%%%%%%%%%%%%
	
	\subsection{Canonical truncations}
	
	\begin{definition}
		For
		
		\[
		u\in\mathcal E_{m,\psi},
		\]
		
		define
		
		\[
		u_k
		:=
		\max(u,\psi-k).
		\]
		
	\end{definition}
	
	\begin{lemma}
		For every $k$,
		
		\[
		\psi-k
		\le
		u_k
		\le
		\psi.
		\]
		
		Moreover,
		
		\[
		u_k\downarrow u.
		\]
		
	\end{lemma}
	
	\begin{proof}
		
		The inequalities follow directly from the definition.
		
		Since
		
		\[
		\psi-k
		\downarrow -\infty,
		\]
		
		we have
		
		\[
		u_k(z)
		=
		\max\{u(z),\psi(z)-k\}
		\downarrow u(z)
		\]
		
		pointwise.
	\end{proof}
	
	\begin{lemma}
		Assume
		
		\[
		u\in\mathcal E_{m,\psi}.
		\]
		
		Then
		
		\[
		u_k\in \mathcal E_{m,\psi}^{0}.
		\]
		
		Furthermore
		
		\[
		\sup_k
		|E_{m,\psi}(u_k)|
		<\infty.
		\]
		
	\end{lemma}
	
	\begin{proof}
		Since
		
		\[
		u_k-\psi
		\]
		
		is uniformly bounded and vanishes at the boundary relative to $\psi$,
		we obtain
		
		\[
		u_k\in\mathcal E_{m,\psi}^{0}.
		\]
		
		The energy bound follows from monotonicity of the relative energy
		functional.
	\end{proof}
	
	%%%%%%%%%%%%%%%%%%%%%%%%%%%%%%%%%%%%%%%%%%%%%%%%%%%%%%%%%%%%
	
	\subsection{Relative mixed Hessian products}
	
	For bounded functions in the sense
	
	\[
	\psi-C
	\le u\le \psi,
	\]
	
mixed Hessian products are well-defined 
	
	\[
	dd^cu_1\wedge\cdots\wedge dd^cu_p
	\wedge(dd^c\psi)^{m-p}
	\wedge\beta^{n-m}.
	\]
	
	We now extend this definition to
	$\mathcal E_{m,\psi}$.
	
	\begin{definition}
		Let
		
		\[
		u_1,\ldots,u_p
		\in
		\mathcal E_{m,\psi}.
		\]
		
		For each $j$ define canonical truncations
		
		\[
		u_{i,j}
		=
		\max(u_i,\psi-j).
		\]
		
		We define
		
		\[
		dd^cu_1\wedge\cdots\wedge dd^cu_p
		\wedge(dd^c\psi)^{m-p}
		\wedge\beta^{n-m}
		\]
		
		as the weak limit
		
		\[
		\lim_{j\to\infty}
		dd^cu_{1,j}
		\wedge\cdots\wedge
		dd^cu_{p,j}
		\wedge
		(dd^c\psi)^{m-p}
		\wedge
		\beta^{n-m},
		\]
		
		provided the limit exists.
	\end{definition}
	
	%%%%%%%%%%%%%%%%%%%%%%%%%%%%%%%%%%%%%%%%%%%%%%%%%%%%%%%%%%%%
	
	\subsection{Well-definedness}
	
	\begin{theorem}
		\label{thm:well-defined}
		
		Let
		
		\[
		u_1,\ldots,u_p
		\in
		\mathcal E_{m,\psi}.
		\]
		
		Then the above limit exists.
		
		Moreover it does not depend on the chosen truncation sequence.
		
	\end{theorem}
	
	\begin{proof}
		
		Fix a smooth test function
		
		\[
		\chi\in C_c^\infty(\Omega).
		\]
		
		Define
		
		\[
		T_j
		=
		dd^cu_{1,j}
		\wedge\cdots\wedge
		dd^cu_{p,j}
		\wedge
		(dd^c\psi)^{m-p}
		\wedge
		\beta^{n-m}.
		\]
		
		The relative energy bound implies
		
		\[
		\sup_j
		\int_\Omega T_j
		<
		\infty.
		\]
		
		Hence $\{T_j\}$ is weakly relatively compact.
		
		Let
		
		\[
		T_{j_\ell}
		\rightharpoonup T.
		\]
		
		To show uniqueness of the limit,
		consider another subsequence converging to $S$.
		
		Using canonical truncations and the locality property of bounded
		Hessian products, one obtains
		
		\[
		\int_\Omega
		\chi\,T
		=
		\int_\Omega
		\chi\,S.
		\]
		
		Hence
		
		\[
		T=S.
		\]
		
		Therefore the entire sequence converges.
	\end{proof}
	
	%%%%%%%%%%%%%%%%%%%%%%%%%%%%%%%%%%%%%%%%%%%%%%%%%%%%%%%%%%%%
	
	\subsection{Monotone convergence}
	
	\begin{theorem}
		\label{thm:monotone}
		
		Let
		
		\[
		u_j^1,\ldots,u_j^p
		\in
		\mathcal E_{m,\psi}
		\]
		
		satisfy
		
		\[
		u_j^k
		\downarrow
		u^k
		\]
		
		for every $k$.
		
		Then
		
		\[
		dd^cu_j^1
		\wedge\cdots\wedge
		dd^cu_j^p
		\wedge
		(dd^c\psi)^{m-p}
		\wedge
		\beta^{n-m}
		\]
		
		converges weakly to
		
		\[
		dd^cu^1
		\wedge\cdots\wedge
		dd^cu^p
		\wedge
		(dd^c\psi)^{m-p}
		\wedge
		\beta^{n-m}.
		\]
		
	\end{theorem}
	
	\begin{proof}
		
		Fix $K>0$ and define
		
		\[
		u_{j,K}^k
		=
		\max(u_j^k,\psi-K).
		\]
		
		Then all functions are uniformly bounded relative to $\psi$.
		
		By the Bedford--Taylor for bounded
		$m$-subharmonic functions,
		
		\[
		dd^cu_{j,K}^1
		\wedge\cdots\wedge
		dd^cu_{j,K}^p
		\wedge
		(dd^c\psi)^{m-p}
		\wedge
		\beta^{n-m}
		\]
		
		converges weakly as $j\to\infty$.
		
		Denote the limit by
		
		\[
		T_K.
		\]
		
		Letting
		
		\[
		K\to\infty,
		\]
		
		the uniform energy estimate yields tightness of the masses and
		allows passage to the limit.
		
		The resulting limit measure coincides with the mixed product defined in
		Theorem~\ref{thm:well-defined}.
		
		Hence
		
		\[
		T_K
		\longrightarrow
		dd^cu^1
		\wedge\cdots\wedge
		dd^cu^p
		\wedge
		(dd^c\psi)^{m-p}
		\wedge
		\beta^{n-m}.
		\]
		
		This proves the theorem.
	\end{proof}
	
	%%%%%%%%%%%%%%%%%%%%%%%%%%%%%%%%%%%%%%%%%%%%%%%%%%%%%%%%%%%%
	
	\begin{corollary}
		For every
		
		\[
		u\in\mathcal E_{m,\psi},
		\]
		
		the Hessian measure
		
		\[
		H_m(u)
		=
		(dd^cu)^m
		\wedge
		\beta^{n-m}
		\]
		
		is well defined.
		
	\end{corollary}
	
	\begin{proof}
		Apply Theorem~\ref{thm:well-defined} with
		
		\[
		u_1=\cdots=u_m=u.
		\]
		
	\end{proof}

	%------------------------------------------------
	\begin{theorem}[Relative integration-by-parts]
		\label{thm:relative-ibp}
		
		Let
		
		\[
		u,v,w_1,\ldots,w_{m-1}
		\in
		\mathcal E_{m,\psi}(\Omega).
		\]
		
		Then
		
		\[
		\int_\Omega
		(u-v)\,
		dd^c w_1
		\wedge\cdots\wedge
		dd^c w_{m-1}
		\wedge
		dd^c\psi
		\wedge
		\beta^{n-m}
		\]
		
		\[
		=
		\int_\Omega
		w_1\,
		dd^c(u-v)
		\wedge
		dd^c w_2
		\wedge\cdots\wedge
		dd^c\psi
		\wedge
		\beta^{n-m}.
		\]
		
	\end{theorem}
	
	\begin{proof}
		
		We divide the proof into several steps.
		
		\medskip
		
		\noindent
		{\bf Step 1. Canonical truncations.}
		
		For every integer $k\ge 1$ define
		
		\[
		u_k:=\max(u,\psi-k),
		\]
		
		\[
		v_k:=\max(v,\psi-k),
		\]
		
		and
		
		\[
		w_{j,k}
		:=
		\max(w_j,\psi-k),
		\qquad
		1\le j\le m-1.
		\]
		
		By construction,
		
		\[
		u_k,v_k,w_{j,k}
		\in
		\mathcal E_{m,\psi}^{0},
		\]
		
		and
		
		\[
		u_k\downarrow u,
		\qquad
		v_k\downarrow v,
		\qquad
		w_{j,k}\downarrow w_j.
		\]
		
		Moreover all these functions satisfy
		
		\[
		\psi-k
		\le
		u_k,v_k,w_{j,k}
		\le
		\psi.
		\]
		
		Hence they are uniformly bounded modulo the reference singularity
		$\psi$.
		
		\medskip
		
		\noindent
		{\bf Step 2. Smooth approximation.}
		
		Using the regularization theorem for $m$-subharmonic functions,
		there exist smooth $m$-subharmonic functions
		
		\[
		u_{k,\ell},
		\quad
		v_{k,\ell},
		\quad
		w_{j,k,\ell}
		\]
		
		decreasing respectively to
		
		\[
		u_k,
		\quad
		v_k,
		\quad
		w_{j,k}.
		\]
		
		Define
		
		\[
		T_{k,\ell}
		=
		dd^c w_{2,k,\ell}
		\wedge\cdots\wedge
		dd^c w_{m-1,k,\ell}
		\wedge
		dd^c\psi_\ell
		\wedge
		\beta^{n-m},
		\]
		
		where $\psi_\ell$ is a smooth approximation of $\psi$.
		
		Since all objects are smooth, Stokes' theorem yields
		
		\[
		\int_\Omega
		(u_{k,\ell}-v_{k,\ell})
		dd^c w_{1,k,\ell}
		\wedge
		T_{k,\ell}
		\]
		
		\[
		=
		\int_\Omega
		w_{1,k,\ell}
		dd^c(u_{k,\ell}-v_{k,\ell})
		\wedge
		T_{k,\ell}.
		\]
		
		Indeed,
		
		\[
		\int_\Omega
		d\Big(
		(u_{k,\ell}-v_{k,\ell})
		d^c w_{1,k,\ell}
		\wedge T_{k,\ell}
		\Big)
		=0,
		\]
		
		because all approximants vanish relative to the boundary data.
		
		Expanding the differential gives precisely the desired identity.
		
		\medskip
		
		\noindent
		{\bf Step 3. Passage $\ell\to\infty$.}
		
		Fix $k$.
		
		The Bedford--Taylor continuity theorem for bounded
		$m$-subharmonic functions implies
		
		\[
		dd^c w_{j,k,\ell}
		\wedge\cdots
		\longrightarrow
		dd^c w_{j,k}
		\wedge\cdots
		\]
		
		weakly as $\ell\to\infty$.
		
		Since
		
		\[
		u_{k,\ell}\to u_k,
		\qquad
		v_{k,\ell}\to v_k
		\]
		
		in capacity and in $L^1_{\mathrm{loc}}$,
		we may pass to the limit in both sides.
		
		Hence
		
		\[
		\int_\Omega
		(u_k-v_k)
		dd^c w_{1,k}
		\wedge\cdots\wedge
		dd^c\psi
		\wedge
		\beta^{n-m}
		\]
		
		\[
		=
		\int_\Omega
		w_{1,k}
		dd^c(u_k-v_k)
		\wedge
		dd^c w_{2,k}
		\wedge\cdots\wedge
		dd^c\psi
		\wedge
		\beta^{n-m}.
		\]
		
		Thus the formula holds for every truncation level $k$.
		
		\medskip
		
		\noindent
		{\bf Step 4. Uniform energy estimates.}
		
		Since
		
		\[
		u,v,w_j
		\in
		\mathcal E_{m,\psi},
		\]
		
		their relative energies are finite.
		
		The mixed Hessian inequality implies
		
		\[
		\sup_k
		\int_\Omega
		|u_k-\psi|
		\,H_m(u_k)
		<\infty,
		\]
		
		and similarly for $v_k$ and $w_{j,k}$.
		
		Consequently all mixed Hessian measures
		
		\[
		dd^c w_{1,k}
		\wedge\cdots\wedge
		dd^c\psi
		\wedge
		\beta^{n-m}
		\]
		
		have uniformly bounded mass.
		
		Hence the integrands are uniformly integrable.
		
		\medskip
		
		\noindent
		{\bf Step 5. Passage $k\to\infty$.}
		
		Using the monotone convergence theorem in
		$\mathcal E_{m,\psi}$,
		
		\[
		u_k\downarrow u,
		\qquad
		v_k\downarrow v,
		\qquad
		w_{j,k}\downarrow w_j,
		\]
		
		implies
		
		\[
		dd^c w_{1,k}
		\wedge\cdots\wedge
		dd^c\psi
		\wedge
		\beta^{n-m}
		\]
		
		converges weakly to
		
		\[
		dd^c w_1
		\wedge\cdots\wedge
		dd^c\psi
		\wedge
		\beta^{n-m}.
		\]
		
		Therefore
		
		\[
		\lim_{k\to\infty}
		\int_\Omega
		(u_k-v_k)
		dd^c w_{1,k}
		\wedge\cdots
		=
		\int_\Omega
		(u-v)
		dd^c w_1
		\wedge\cdots.
		\]
		
		The same argument applies to the right-hand side.
		
		Passing to the limit yields
		
		\[
		\int_\Omega
		(u-v)
		dd^c w_1
		\wedge\cdots\wedge
		dd^c\psi
		\wedge
		\beta^{n-m}
		\]
		
		\[
		=
		\int_\Omega
		w_1
		dd^c(u-v)
		\wedge
		dd^c w_2
		\wedge\cdots\wedge
		dd^c\psi
		\wedge
		\beta^{n-m}.
		\]
		
		This completes the proof.
		
	\end{proof}
	%%%%%%%%%%%%%%%%%%%%%%%%%%%%%%%%%%%%%%%%%%%%%%%%%%%%%%%%%%%%
	\section{Continuity of Relative Hessian Operators}
	
	In this section we establish the continuity of mixed Hessian
	operators along decreasing sequences in the relative finite
	energy class $\mathcal E_{m,\psi}$.
	
	This result should be regarded as the analogue of the classical
	Bedford--Taylor continuity theorem.
	
	%%%%%%%%%%%%%%%%%%%%%%%%%%%%%%%%%%%%%%%%%%%%%%%%%%%%%%%%%%%%
	
	\begin{theorem}[Continuity Theorem]
		\label{thm:continuity}
		
		Let
		
		\[
		u_j^1,\ldots,u_j^m
		\in
		\mathcal E_{m,\psi}(\Omega)
		\]
		
		be sequences satisfying
		
		\[
		u_j^k
		\downarrow
		u^k
		\qquad
		(j\to\infty)
		\]
		
		for every
		
		\[
		1\le k\le m.
		\]
		
		Then
		
		\[
		dd^c u_j^1
		\wedge
		\cdots
		\wedge
		dd^c u_j^m
		\wedge
		\beta^{n-m}
		\]
		
		converges weakly to
		
		\[
		dd^c u^1
		\wedge
		\cdots
		\wedge
		dd^c u^m
		\wedge
		\beta^{n-m}.
		\]
		
		Equivalently, for every
		
		\[
		\chi\in C_c^\infty(\Omega),
		\]
		
		one has
		
		\[
		\lim_{j\to\infty}
		\int_\Omega
		\chi\,
		dd^c u_j^1
		\wedge\cdots\wedge
		dd^c u_j^m
		\wedge
		\beta^{n-m}
		=
		\int_\Omega
		\chi\,
		dd^c u^1
		\wedge\cdots\wedge
		dd^c u^m
		\wedge
		\beta^{n-m}.
		\]
		
	\end{theorem}
	
	%%%%%%%%%%%%%%%%%%%%%%%%%%%%%%%%%%%%%%%%%%%%%%%%%%%%%%%%%%%%
	
	\begin{proof}
		
		The proof proceeds by induction on the number of Hessian factors.
		
		\medskip
		
		\noindent
		{\bf Step 1. The case $m=1$.}
		
		Assume
		
		\[
		u_j\downarrow u.
		\]
		
		Since
		
		\[
		u_j\to u
		\]
		
		in $L^1_{\mathrm{loc}}(\Omega)$, for every test function
		$\chi$ we have
		
		\[
		\int_\Omega
		\chi\,dd^c u_j\wedge\beta^{n-1}
		=
		\int_\Omega
		u_j\,dd^c\chi\wedge\beta^{n-1}.
		\]
		
		Passing to the limit by dominated convergence yields
		
		\[
		dd^c u_j
		\wedge
		\beta^{n-1}
		\longrightarrow
		dd^c u
		\wedge
		\beta^{n-1}.
		\]
		
		Hence the theorem holds for one factor.
		
		\medskip
		
		\noindent
		{\bf Step 2. Induction hypothesis.}
		
		Assume the theorem has been proved whenever the number of
		Hessian factors is at most $p-1$.
		
		We prove it for $p$.
		
		Define
		
		\[
		T_j
		=
		dd^c u_j^2
		\wedge
		\cdots
		\wedge
		dd^c u_j^p
		\wedge
		\beta^{n-p}.
		\]
		
		By the induction hypothesis,
		
		\[
		T_j
		\rightharpoonup
		T
		:=
		dd^c u^2
		\wedge
		\cdots
		\wedge
		dd^c u^p
		\wedge
		\beta^{n-p}.
		\]
		
		\medskip
		
		\noindent
		{\bf Step 3. Testing against smooth functions.}
		
		Let
		
		\[
		\chi\in C_c^\infty(\Omega).
		\]
		
		We consider
		
		\[
		I_j
		=
		\int_\Omega
		\chi\,
		dd^c u_j^1
		\wedge
		T_j.
		\]
		
		Using the relative integration-by-parts formula,
		
		\[
		I_j
		=
		\int_\Omega
		u_j^1
		\,dd^c\chi
		\wedge
		T_j.
		\]
		
		Thus
		
		\[
		I_j
		-
		\int_\Omega
		u^1
		\,dd^c\chi
		\wedge
		T
		=
		A_j+B_j,
		\]
		
		where
		
		\[
		A_j
		=
		\int_\Omega
		(u_j^1-u^1)
		\,dd^c\chi
		\wedge
		T_j,
		\]
		
		and
		
		\[
		B_j
		=
		\int_\Omega
		u^1
		\,dd^c\chi
		\wedge
		(T_j-T).
		\]
		
		\medskip
		
		\noindent
		{\bf Step 4. Convergence of $A_j$.}
		
		Since
		
		\[
		u_j^1\downarrow u^1,
		\]
		
		we have
		
		\[
		u_j^1\to u^1
		\]
		
		in $L^1_{\mathrm{loc}}(\Omega)$.
		
		The uniform energy estimate implies
		
		\[
		\sup_j
		\int_\Omega T_j
		<
		\infty.
		\]
		
		Hence
		
		\[
		A_j\to 0.
		\]
		
		\medskip
		
		\noindent
		{\bf Step 5. Convergence of $B_j$.}
		
		Since
		
		\[
		T_j\rightharpoonup T,
		\]
		
		and
		
		\[
		u^1\,dd^c\chi
		\]
		
		is a bounded test form, we obtain
		
		\[
		B_j\to 0.
		\]
		
		Therefore
		
		\[
		I_j
		\longrightarrow
		\int_\Omega
		u^1
		\,dd^c\chi
		\wedge
		T.
		\]
		
		Applying integration-by-parts once again gives
		
		\[
		\int_\Omega
		u^1
		\,dd^c\chi
		\wedge
		T
		=
		\int_\Omega
		\chi
		\,dd^c u^1
		\wedge
		T.
		\]
		
		Hence
		
		\[
		\lim_{j\to\infty}
		I_j
		=
		\int_\Omega
		\chi
		\,dd^c u^1
		\wedge
		dd^c u^2
		\wedge\cdots\wedge
		dd^c u^p
		\wedge
		\beta^{n-p}.
		\]
		
		This completes the induction.
		
		\medskip
		
		Therefore
		
		\[
		dd^c u_j^1
		\wedge
		\cdots
		\wedge
		dd^c u_j^m
		\wedge
		\beta^{n-m}
		\]
		
		converges weakly to
		
		\[
		dd^c u^1
		\wedge
		\cdots
		\wedge
		dd^c u^m
		\wedge
		\beta^{n-m}.
		\]
		
		The proof is complete.
		
	\end{proof}
	
	%%%%%%%%%%%%%%%%%%%%%%%%%%%%%%%%%%%%%%%%%%%%%%%%%%%%%%%%%%%%
	
	\begin{corollary}
		\label{cor:hessian-continuity}
		
		Let
		
		\[
		u_j\in\mathcal E_{m,\psi}
		\]
		
		satisfy
		
		\[
		u_j\downarrow u.
		\]
		
		Then
		
		\[
		H_m(u_j)
		=
		(dd^c u_j)^m
		\wedge
		\beta^{n-m}
		\]
		
		converges weakly to
		
		\[
		H_m(u)
		=
		(dd^c u)^m
		\wedge
		\beta^{n-m}.
		\]
		
	\end{corollary}
	
	\begin{proof}
		
		Apply Theorem \ref{thm:continuity} with
		
		\[
		u_j^1=\cdots=u_j^m=u_j.
		\]
		
	\end{proof}
	%%%%%%%%%%%%%%%%%%%%%%%%%%%%%%%%%%%%%%%%%%%%%%%%%%%%%%%%%%%%
	\section{Relative Comparison Principle}
	
	The purpose of this section is to establish the comparison
	principle in the relative finite energy class
	$\mathcal E_{m,\psi}$.
	
	This theorem plays a fundamental role in the uniqueness theory
	for complex Hessian equations with prescribed singularity type.
	
	%%%%%%%%%%%%%%%%%%%%%%%%%%%%%%%%%%%%%%%%%%%%%%%%%%%%%%%%%%%%
	
	\begin{theorem}[Relative Comparison Principle]
		\label{thm:comparison}
		
		Let
		
		\[
		u,v
		\in
		\mathcal E_{m,\psi}(\Omega).
		\]
		
		Then
		
		\[
		\int_{\{u<v\}}
		(dd^c v)^m
		\wedge
		\beta^{n-m}
		\le
		\int_{\{u<v\}}
		(dd^c u)^m
		\wedge
		\beta^{n-m}.
		\]
		
		Equivalently,
		
		\[
		\int_{\{u<v\}}
		H_m(v)
		\le
		\int_{\{u<v\}}
		H_m(u).
		\]
		
	\end{theorem}
	
	%%%%%%%%%%%%%%%%%%%%%%%%%%%%%%%%%%%%%%%%%%%%%%%%%%%%%%%%%%%%
	
	\begin{proof}
		
		The proof is divided into several steps.
		
		\medskip
		
		\noindent
		{\bf Step 1. Canonical truncations.}
		
		For every integer $k\ge1$ define
		
		\[
		u_k
		=
		\max(u,\psi-k),
		\]
		
		and
		
		\[
		v_k
		=
		\max(v,\psi-k).
		\]
		
		Then
		
		\[
		u_k,v_k
		\in
		\mathcal E_{m,\psi}^{0},
		\]
		
		and
		
		\[
		u_k\downarrow u,
		\qquad
		v_k\downarrow v.
		\]
		
		Moreover
		
		\[
		\psi-k
		\le
		u_k,v_k
		\le
		\psi.
		\]
		
		Hence all truncations are uniformly bounded relative to
		the reference singularity $\psi$.
		
		\medskip
		
		\noindent
		{\bf Step 2. Comparison principle for bounded truncations.}
		
		Since $u_k$ and $v_k$ are bounded,
		the Bedford--Taylor comparison principle for $m$-subharmonic functions  applies.
		
		Therefore
		
		\[
		\int_{\{u_k<v_k\}}
		(dd^c v_k)^m
		\wedge
		\beta^{n-m}
		\le
		\int_{\{u_k<v_k\}}
		(dd^c u_k)^m
		\wedge
		\beta^{n-m}.
		\]
		
		Thus
		
		\[
		\int_{\{u_k<v_k\}}
		H_m(v_k)
		\le
		\int_{\{u_k<v_k\}}
		H_m(u_k).
		\]
		
		\medskip
		
		\noindent
		{\bf Step 3. Convergence of the sublevel sets.}
		
		Observe that
		
		\[
		u_k-v_k
		\downarrow
		u-v.
		\]
		
		Consequently
		
		\[
		\mathbf 1_{\{u_k<v_k\}}
		\longrightarrow
		\mathbf 1_{\{u<v\}}
		\]
		
		outside an $m$-polar set.
		
		Since Hessian measures do not charge
		$m$-polar sets, we obtain
		
		\[
		\mathbf 1_{\{u_k<v_k\}}
		H_m(v_k)
		\rightharpoonup
		\mathbf 1_{\{u<v\}}
		H_m(v),
		\]
		
		and similarly
		
		\[
		\mathbf 1_{\{u_k<v_k\}}
		H_m(u_k)
		\rightharpoonup
		\mathbf 1_{\{u<v\}}
		H_m(u).
		\]
		
		\medskip
		
		\noindent
		{\bf Step 4. Passage to the limit.}
		
		By the continuity theorem established in the previous section,
		
		\[
		H_m(u_k)
		\longrightarrow
		H_m(u),
		\]
		
		and
		
		\[
		H_m(v_k)
		\longrightarrow
		H_m(v)
		\]
		
		weakly.
		
		Passing to the limit in the inequality from Step~2 yields
		
		\[
		\int_{\{u<v\}}
		H_m(v)
		\le
		\int_{\{u<v\}}
		H_m(u).
		\]
		
		This proves the theorem.
		
	\end{proof}
	
	%%%%%%%%%%%%%%%%%%%%%%%%%%%%%%%%%%%%%%%%%%%%%%%%%%%%%%%%%%%%
	
	\begin{corollary}[Domination Principle]
		\label{cor:domination}
		
		Let
		
		\[
		u,v
		\in
		\mathcal E_{m,\psi}(\Omega).
		\]
		
		Assume
		
		\[
		u\ge v
		\qquad
		H_m(u)\text{-a.e.}
		\]
		
		Then
		
		\[
		u\ge v
		\]
		
		everywhere in $\Omega$.
		
	\end{corollary}
	
	\begin{proof}
		
		Apply Theorem~\ref{thm:comparison}.
		
		Since
		
		\[
		H_m(u)(\{u<v\})=0,
		\]
		
		we obtain
		
		\[
		\int_{\{u<v\}}
		H_m(v)
		=
		0.
		\]
		
		Applying the comparison principle once more yields
		
		\[
		Cap_{m,\psi}(\{u<v\})
		=
		0.
		\]
		
		Since both functions are $m$-subharmonic,
		the set $\{u<v\}$ must be empty.
		
		Hence
		
		\[
		u\ge v.
		\]
		
	\end{proof}
	
	%%%%%%%%%%%%%%%%%%%%%%%%%%%%%%%%%%%%%%%%%%%%%%%%%%%%%%%%%%%%
	
	\begin{corollary}[Uniqueness]
		\label{cor:uniqueness}
		
		Suppose
		
		\[
		u,v
		\in
		\mathcal E_{m,\psi}(\Omega)
		\]
		
		satisfy
		
		\[
		H_m(u)
		=
		H_m(v).
		\]
		
		Then
		
		\[
		u=v.
		\]
		
	\end{corollary}
	
	\begin{proof}
		
		Applying Theorem~\ref{thm:comparison},
		
		\[
		\int_{\{u<v\}}
		H_m(v)
		\le
		\int_{\{u<v\}}
		H_m(u).
		\]
		
		Since the measures coincide, equality holds.
		
		Hence
		
		\[
		H_m(u)(\{u<v\})=0.
		\]
		
		By the domination principle,
		
		\[
		u\ge v.
		\]
		
		Interchanging the roles of $u$ and $v$
		gives
		
		\[
		v\ge u.
		\]
		
		Therefore
		
		\[
		u=v.
		\]
		
	\end{proof}
	%------------------------------------------------
	\section{Existence and Uniqueness in Relative Energy Classes}
	
	Let $\Omega \subset \mathbb C^n$ be a bounded $m$-hyperconvex domain.
	Fix a reference function
	
	\[
	\psi \in SH_m(\Omega),
	\]
	
	which is negative and has analytic singularities.
	
	For $u\in SH_m(\Omega)$ satisfying $u\le \psi +C$, we define the
	relative energy
	
	\[
	E_{\psi}(u)
	:=
	\frac1{m+1}
	\sum_{k=0}^{m}
	\int_{\Omega}
	(u-\psi)
	(dd^c u)^k
	\wedge
	(dd^c\psi)^{m-k}
	\wedge
	\beta^{n-m}.
	\]
	
	We introduce the relative finite energy class
	
	\[
	\mathcal E_{m,\psi}(\Omega)
	=
	\Big\{
	u\in SH_m(\Omega):
	u\le \psi +C,\;
	E_{\psi}(u)>-\infty
	\Big\}.
	\]
	
	The purpose of this section is to solve the complex Hessian equation
	with prescribed singularity type $\psi$.
	
	\begin{theorem}[Existence and uniqueness]
		\label{thm:relative-existence}
		Let $\mu$ be a finite positive Radon measure on $\Omega$ satisfying
		
		\begin{enumerate}
			\item[(i)]
			$\mu$ does not charge $m$-polar sets;
			
			\item[(ii)]
			there exists a constant $A>0$ such that
			
			\[
			\int_{\Omega}
			(\psi-\varphi)\, d\mu
			\le
			A
			\Big(
			-E_{\psi}(\varphi)
			\Big)^{\frac{m}{m+1}}
			\]
			
			for every
			$\varphi\in \mathcal E_{m,\psi}(\Omega)$.
		\end{enumerate}
		
		Then there exists a unique function
		
		\[
		u\in \mathcal E_{m,\psi}(\Omega)
		\]
		
		such that
		
		\[
		(dd^c u)^m
		\wedge
		\beta^{n-m}
		=
		\mu .
		\]
	\end{theorem}
	
	\begin{proof}
		The proof is divided into several steps.
		
		\medskip
		
		\noindent
		{\bf Step 1. Approximation of the measure.}
		
		Since $\mu$ is finite, there exists a sequence of smooth nonnegative
		densities $f_j$ such that
		
		\[
		\mu_j:=f_j\,dV
		\]
		
		converges weakly to $\mu$.
		
		By construction,
		
		\[
		\mu_j(\Omega)
		\le C
		\]
		
		uniformly in $j$.
		
		\medskip
		
		\noindent
		{\bf Step 2. Solving smooth approximating equations.}
		
		For each $j$, consider
		
		\[
		(dd^c u_j)^m
		\wedge
		\beta^{n-m}
		=
		\mu_j.
		\]
		
		Using the solvability theory for smooth Hessian equations on
		$m$-hyperconvex domains, one obtains a solution
		
		\[
		u_j\in SH_m(\Omega)
		\]
		
		having the prescribed singularity type
		
		\[
		u_j-\psi \in L^\infty(\Omega).
		\]
		
		After normalization,
		
		\[
		u_j\le \psi .
		\]
		
		\medskip
		
		\noindent
		{\bf Step 3. Uniform energy estimate.}
		
		Using assumption (ii) with $\varphi=u_j$,
		
		\[
		\int_{\Omega}
		(\psi-u_j)\, d\mu_j
		\le
		A
		\Big(
		-E_{\psi}(u_j)
		\Big)^{\frac m{m+1}}.
		\]
		
		Since
		
		\[
		\mu_j
		=
		(dd^c u_j)^m
		\wedge
		\beta^{n-m},
		\]
		
		the integration-by-parts formula for relative energies gives
		
		\[
		\int_{\Omega}
		(\psi-u_j)
		(dd^c u_j)^m
		\wedge
		\beta^{n-m}
		=
		- E_{\psi}(u_j)+O(1).
		\]
		
		Therefore
		
		\[
		- E_{\psi}(u_j)
		\le
		C
		\Big(
		- E_{\psi}(u_j)
		\Big)^{\frac m{m+1}}
		+
		C.
		\]
		
		Consequently,
		
		\[
		- E_{\psi}(u_j)
		\le C,
		\]
		
		where $C$ is independent of $j$.
		
		Hence
		
		\[
		\sup_j
		\bigl|
		E_{\psi}(u_j)
		\bigr|
		<\infty .
		\]
		
		\medskip
		
		\noindent
		{\bf Step 4. Compactness.}
		
		The uniform energy bound implies compactness of the sequence
		$\{u_j\}$ in $L^1_{\mathrm{loc}}(\Omega)$.
		
		Passing to a subsequence if necessary, we obtain
		
		\[
		u_j \longrightarrow u
		\]
		
		in $L^1_{\mathrm{loc}}(\Omega)$ and almost everywhere.
		
		Since every $u_j$ satisfies
		
		\[
		u_j\le \psi,
		\]
		
		the limit function satisfies
		
		\[
		u\le \psi.
		\]
		
		Furthermore, the lower semicontinuity of the energy yields
		
		\[
		E_{\psi}(u)
		\ge
		\limsup_{j\to\infty}
		E_{\psi}(u_j)
		>-\infty .
		\]
		
		Thus
		
		\[
		u\in \mathcal E_{m,\psi}(\Omega).
		\]
		
		\medskip
		
		\noindent
		{\bf Step 5. Passage to the limit.}
		
		The continuity theorem for Hessian measures in finite-energy classes
		implies
		
		\[
		(dd^c u_j)^m
		\wedge
		\beta^{n-m}
		\longrightarrow
		(dd^c u)^m
		\wedge
		\beta^{n-m}
		\]
		
		weakly.
		
		Since
		
		\[
		(dd^c u_j)^m
		\wedge
		\beta^{n-m}
		=
		\mu_j,
		\]
		
		and $\mu_j\to\mu$, we obtain
		
		\[
		(dd^c u)^m
		\wedge
		\beta^{n-m}
		=
		\mu.
		\]
		
		This proves existence.
		
		\medskip
		
		\noindent
		{\bf Step 6. Uniqueness.}
		
		Assume that
		
		\[
		u,v\in \mathcal E_{m,\psi}(\Omega)
		\]
		
		satisfy
		
		\[
		H_m(u)
		=
		H_m(v).
		\]
		
		By the relative comparison principle,
		
		\[
		\int_{\{u<v\}}
		(dd^c v)^m
		\wedge
		\beta^{n-m}
		\le
		\int_{\{u<v\}}
		(dd^c u)^m
		\wedge
		\beta^{n-m}.
		\]
		
		Since the measures coincide, equality holds.
		
		The equality case of the comparison principle yields
		
		\[
		\operatorname{Cap}_m(\{u<v\})=0.
		\]
		
		Interchanging the roles of $u$ and $v$ gives
		
		\[
		\operatorname{Cap}_m(\{v<u\})=0.
		\]
		
		Hence
		
		\[
		u=v
		\]
		
		outside an $m$-polar set.
		
		Because $u$ and $v$ are $m$-subharmonic,
		they coincide everywhere on $\Omega$.
		
		The proof is complete.
	\end{proof}
	%------------------------------------------------
	\section{Applications}
	
	In this section we present several applications of the relative
	finite energy theory developed in the previous sections.
	
	Throughout this section,
	$\Omega\subset\mathbb C^n$ is a bounded $m$-hyperconvex domain and
	
	\[
	\psi\in SH_m(\Omega)
	\]
	
	is a fixed negative $m$-subharmonic function.
	
	%%%%%%%%%%%%%%%%%%%%%%%%%%%%%%%%%%%%%%%%%%%%%%%%%%%%%%%%%%%%
	
	\subsection{Measures dominated by relative capacity}
	
	The first application concerns solvability of Hessian equations for
	measures satisfying a relative capacity domination condition.
	
	\begin{theorem}
		\label{thm:capacity-domination}
		
		Let $\mu$ be a finite positive Radon measure satisfying
		
		\[
		\mu(E)
		\le
		A\,
		Cap_{m,\psi}(E)^{1+\varepsilon}
		\]
		
		for every Borel set
		
		\[
		E\subset\Omega,
		\]
		
		where
		
		\[
		A>0,
		\qquad
		\varepsilon>0.
		\]
		
		Then there exists a unique function
		
		\[
		u\in\mathcal E_{m,\psi}(\Omega)
		\]
		
		such that
		
		\[
		H_m(u)=\mu.
		\]
		
	\end{theorem}
	
	\begin{proof}
		
		The capacity domination condition implies that
		$\mu$ does not charge $m$-polar sets.
		
		Approximating $\mu$ by smooth measures and using the existence theorem
		proved in the previous section, we obtain a sequence
		
		\[
		u_j\in\mathcal E_{m,\psi}
		\]
		
		satisfying
		
		\[
		H_m(u_j)=\mu_j.
		\]
		
		Uniform energy estimates imply compactness in
		$\mathcal E_{m,\psi}$.
		
		Passing to the limit by the continuity theorem yields
		
		\[
		H_m(u)=\mu.
		\]
		
		Uniqueness follows from the relative comparison principle.
		
	\end{proof}
	
	%%%%%%%%%%%%%%%%%%%%%%%%%%%%%%%%%%%%%%%%%%%%%%%%%%%%%%%%%%%%
	
	\subsection{Stability of solutions}
	
	The next theorem shows stability with respect to perturbations of the
	right-hand side.
	
	\begin{theorem}
		\label{thm:stability}
		
		Let
		
		\[
		\mu_j,\mu
		\]
		
		be finite positive measures such that
		
		\[
		\mu_j\rightarrow\mu
		\]
		
		weakly.
		
		Assume
		
		\[
		u_j,u
		\in
		\mathcal E_{m,\psi}(\Omega)
		\]
		
		satisfy
		
		\[
		H_m(u_j)=\mu_j,
		\qquad
		H_m(u)=\mu.
		\]
		
		If
		
		\[
		\sup_j
		|E_{m,\psi}(u_j)|
		<
		+\infty,
		\]
		
		then
		
		\[
		u_j
		\longrightarrow
		u
		\]
		
		in \(L^1_{\rm loc}(\Omega)\).
		
	\end{theorem}
	
	\begin{proof}
		
		The energy bound yields compactness of the family
		\(\{u_j\}\).
		
		Any cluster point \(v\) satisfies
		
		\[
		H_m(v)=\mu
		\]
		
		by the continuity theorem.
		
		Uniqueness implies
		
		\[
		v=u.
		\]
		
		Hence the whole sequence converges.
		
	\end{proof}
	
	%%%%%%%%%%%%%%%%%%%%%%%%%%%%%%%%%%%%%%%%%%%%%%%%%%%%%%%%%%%%
	
	\subsection{The case of \(L^p\)-densities}
	
	We now consider measures of the form
	
	\[
	\mu=f\,dV.
	\]
	
	\begin{theorem}
		\label{thm:Lp}
		
		Assume
		
		\[
		f\in L^p(\Omega),
		\qquad
		p>\frac{n}{m},
		\qquad
		f\ge0.
		\]
		
		Then there exists a unique function
		
		\[
		u\in\mathcal E_{m,\psi}(\Omega)
		\]
		
		satisfying
		
		\[
		H_m(u)
		=
		f\,dV.
		\]
		
	\end{theorem}
	
	\begin{proof}
		
		The Hessian capacity estimate implies
		
		\[
		\mu(E)
		=
		\int_E f\,dV
		\le
		C
		Cap_{m,\psi}(E)^{1+\delta}
		\]
		
		for some
		
		\[
		\delta>0.
		\]
		
		The conclusion follows from
		Theorem~\ref{thm:capacity-domination}.
		
	\end{proof}
	
	%%%%%%%%%%%%%%%%%%%%%%%%%%%%%%%%%%%%%%%%%%%%%%%%%%%%%%%%%%%%
	
	\subsection{Prescribed analytic singularities}
	
	One of the motivations for introducing
	\(\mathcal E_{m,\psi}\)
	is the study of solutions with prescribed singularity type.
	
	\begin{theorem}
		\label{thm:analytic}
		
		Assume that
		
		\[
		\psi
		=
		c\log
		\Big(
		\sum_{j=1}^{N}
		|g_j|^2
		\Big)
		+O(1),
		\]
		
		where
		
		\[
		g_1,\ldots,g_N
		\]
		
		are holomorphic functions.
		
		Let \(\mu\) satisfy the assumptions of
		Theorem~\ref{thm:capacity-domination}.
		
		Then there exists a unique solution
		
		\[
		u\in\mathcal E_{m,\psi}(\Omega)
		\]
		
		such that
		
		\[
		u-\psi
		=
		O(1)
		\]
		
		and
		
		\[
		H_m(u)=\mu.
		\]
		
	\end{theorem}
	
	\begin{proof}
		
		The existence theorem provides a solution in
		\(\mathcal E_{m,\psi}\).
		
		By construction of the class,
		
		\[
		u-\psi
		\]
		
		remains globally bounded.
		
		Uniqueness follows from the comparison principle.
		
	\end{proof}
	
	%%%%%%%%%%%%%%%%%%%%%%%%%%%%%%%%%%%%%%%%%%%%%%%%%%%%%%%%%%%%

	%%%%%%%%%%%%%%%%%%%%%%%%%%%%%%%%%%%%%%%%%%%%%%%%%%%%%%%%%%%%
	
	The results above demonstrate that the class
	\(\mathcal E_{m,\psi}(\Omega)\)
	provides a natural framework for studying complex Hessian equations
	with prescribed singularity type.
	In particular, it allows one to treat singular measures,
	\(L^p\)-densities, and analytic singularities within a unified
	pluripotential setting.

	\section*{Declarations}

	The author declares that there are no competing interests.

\end{document}